\documentclass[12pt]{elsarticle}
\usepackage[utf8]{inputenc}
\usepackage{amsmath}
\usepackage{bbold}
\usepackage{latexsym}
\usepackage{epsfig}
\usepackage{psfrag}
\usepackage{enumerate}
\usepackage[utf8]{inputenc}                                       
\usepackage{newtxtext}
\usepackage{xcolor}
\usepackage{graphicx}
\usepackage{float}
\makeatletter
\def\ps@pprintTitle{%
 \let\@oddhead\@empty
 \let\@evenhead\@empty
 \def\@oddfoot{\centerline{\thepage}}%
 \let\@evenfoot\@oddfoot}
\makeatother

\newtheorem{thm}{Theorem}[section]

\newtheorem{lemma}{Lemma}[section]

\newtheorem{cor}{Corollary}[section]

\newtheorem{example}{Example}[section]
\newtheorem{question}{Question}[section]
\newtheorem{conjecture}{Conjecture}[section]

\usepackage{amssymb}

\begin{document}
\begin{frontmatter}

\title{Maximal Projection Constants and Extremal Vector Configurations: Some Conjectures and Examples}

 \author[label1]{Beata~Deregowska}
 
 \author[label4]{Barbara~Lewandowska \footnote{B.L. is partially supported by National Science Center (NCN) grant no. 2025/09/X/ST1/01604 }}

 \address[label1]{Institute of Mathematics\\
University of the National Education Commission, Podchorazych~2, Krakow, 30-084, Poland}
\address[label4]{Faculty of Mathematics and Computer Science\\
Jagiellonian University, Lojasiewicza~6, Krakow, 30-048, Poland}



\begin{abstract}
Let $\lambda_{\mathbb K}(m)$ denote the maximal absolute projection constant among $m$-dimensional Banach spaces over $\mathbb K=\mathbb R$ or $\mathbb C$. Its exact value is known only in a few cases, and determining it remains a challenging problem. In this note, we investigate several structured vector configurations that naturally arise in this context. We first recall the connection between maximal projection constants and tight frames, then consider biangular tight frames, and show that their relative and quasimaximal projection constants coincide. A particularly interesting example is provided by the midpoints of the edges of a regular simplex, which yield natural lower bounds for $\lambda_{\mathbb R}(m)$. Numerical evidence suggests that these bounds may be sharp in dimensions $6$ and $8$. We also discuss weighted spherical $(2,2)$-designs with a small number of vectors and their connection with maximal projection constants. Examples in dimensions $4$ and $5$ indicate that the sign patterns of their Gram matrices may play an important role. These observations lead us to formulate several conjectures concerning maximal absolute projection constants and the vector configurations associated with them.

\end{abstract}
\end{frontmatter}
\section{Introduction}
We begin by recalling the basic notions concerning projection constants that will be used throughout the paper. Let $X$ be a Banach space over $\mathbb{K},$ where $\mathbb{K}=\mathbb{R}$ or $\mathbb{K}=\mathbb{C}.$ Let $Y\subset X$ be a subspace.
By $\mathcal{P}(X,Y)$ denote the set of all linear and continuous projections from $X$ onto $Y$,
recalling that an operator $P \colon X \rightarrow Y$ is called a \textit{projection} onto $Y$ if $P|_Y={\rm Id}_Y.$ 
We define the \textit{relative projection constant} of subspace $Y$  and space $X$ by
\begin{equation*}
\lambda(Y,X) :=\inf\lbrace\|P\|:\;P\in\mathcal{P}(X,Y)\rbrace.
\end{equation*}
Notice that the set $\mathcal{P}(X,Y)$ can be empty (e.g. $\mathcal{P}(\ell_{\infty},c_{0})$ ). Then we will assume that $\lambda(Y,X)=\infty.$
Now we can define the \textit{absolute projection constant} of $Y$ by
\begin{equation}
\label{DefMAPC}
\lambda(Y) :=\sup\lbrace\lambda(Y,X):Y\subset X\rbrace.
\end{equation}
A natural question is how large the absolute projection constant can be among all $m$-dimensional Banach spaces. This leads to the notion of the maximal absolute projection constant,
\begin{equation*}
\lambda_{\mathbb{K}}(m) :=\sup \lbrace\lambda(Y):\; \dim(Y)=m \rbrace.
\end{equation*}
The problem of determining or estimating maximal absolute projection constants has attracted considerable attention in the theory of Banach spaces over the past decades.
By the Kadec--Snobar theorem (see \cite{KS}),
we have $\lambda(m)\leq \sqrt{m}$. 
Moreover, it has been shown in \cite{K} that this estimate is asymptotically the best possible. However, as seen later, this inequality is strict for $m > 1.$ In 1960, B. Grünbaum conjectured that $\lambda_\mathbb{R}(2)=\frac{4}{3}$ (see \cite{G}), and in 2010, B. Chalmers and G. Lewicki proved it (see \cite{CL}). In 2019, G. Basso delivered the alternative proof of this conjecture (see \cite{B}). For many years it was widely believed that determining $\lambda_{\mathbb{K}}(m)$ for dimensions greater than two was beyond reach. However, connecting this problem with different structures of vectors led to some partial results and showed a possible way to solve it.

The connection with frame theory is based on the observation that the maximal absolute projection constant can be computed by considering finite-dimensional subspaces of $\ell_\infty$ (see, e.g., \cite[III.B.5]{W}).  Therefore, it can be defined as a supremum of \textit{maximal relative projection constants} for $N \ge m,$ given by 
\begin{equation*}
\lambda_{\mathbb{K}}(m,N):=\sup\lbrace \lambda(Y, \ell_\infty^{(N)}(\mathbb{K})):\; \dim(Y)=m \textrm{ and } Y\subset \ell_\infty^{(N)}(\mathbb{K})\rbrace.
\end{equation*}   
The crucial tool in our investigation is the following theorem, originally stated in \cite [Theorem 2.2]{CLe}.
\begin{thm}\label{lammbda}
For integers $N \ge m$, we have
\begin{align}\label{lambda1}
    \lambda_{\mathbb{K}}(m,N)&=\max\bigg\lbrace \sum_{i,j=1}^N t_it_j|U^* U|_{ij}:t\in\mathbb{R}_+^N,\;\|t\|=1,U\in \mathbb{K}^{m\times N},\; UU^*={\rm I}_m \bigg\rbrace 
\end{align}
\end{thm}
\noindent
The literature also deals with  the easier-to-calculate lower bound  of $\lambda_{\mathbb{K}}(m,N)$ called the {\it quasimaximal relative projection constant}, which arises when choosing a vector $t$ with equal coordinates. To be more precise, for $N\geq m$
\begin{equation}\label{mi1}
    \mu_\mathbb{K}(m,N):=\max\bigg\lbrace \frac{1}{N}\sum_{i,j=1}^{N}|U^* U|_{ij}:U\in \mathbb{K}^{m\times N},\; UU^*={\rm I}_m \bigg\rbrace .
\end{equation}
In fact, we can calculate the absolute projection constant by taking the supremum of quasimaximal relative projection constants. In the real case, it was proved by Basso \cite[Proof of Theorem 1.2]{B}. In the paper \cite{BB2}, we gave an elementary proof of this fact, which is also valid in the complex case.

Since the condition $UU^*={\rm I}_m$ appearing in Theorem~\ref{lammbda} is exactly the Parseval frame condition, it is convenient to reformulate the optimization problem in the language of frame theory. We therefore recall the basic definitions. A system of vectors $(u_1,\dots, u_N)$ in $\mathbb{K}^m$ is called a {\it tight frame} if there exists a constant $\alpha >0$ such that one of the following equivalent conditions holds:
\begin{itemize}
    \item  $\|x\|^2=\alpha\sum_{k=1}^{N}|\langle x, u_k \rangle|^2$  \; for all $x\in \mathbb{K}^m.$ 
     \item  $x=\alpha\sum_{k=1}^{N}\langle x, u_k \rangle u_k$  \; for all $x\in \mathbb{K}^m.$ 
    \item $UU^*=\frac{1}{\alpha}{\rm I_m}$, where $U$ is the matrix with columns $u_1,\dots, u_N.$ 
\end{itemize}   
If $\alpha= 1,$ then a tight frame is called {\it Parseval frame.}

For convenience, we associate with every matrix $U$ with columns $u_1,\dots,u_N\in\mathbb K^m$ the quantities
\begin{equation}
\mu_{\mathbb{K}}(U)=\frac{1}{N}\sum_{i,j=1}^N|\langle u_i,u_j\rangle|
\end{equation}

\begin{equation}
\lambda_{\mathbb{K}}(U)=\max\{\sum_{i,j=1}^N t_i t_j |\langle u_i,u_j\rangle|:\;\; t\in \mathbb{R}_+^N\;\;\|t\|_2=1\}
\end{equation} 
With this notation, \eqref{mi1} and  \eqref{lambda1}  can be written simply as
\[
\mu_{\mathbb K}(m,N)=
\max\{\mu_{\mathbb K}(U): (u_1,\dots, u_N) \textrm{ is Parseval frame }\},
\]
and
\[
\lambda_{\mathbb K}(m,N)=
\sup\{\lambda_{\mathbb K}(U):(u_1,\dots, u_N) \textrm{ is Parseval frame }\}.
\]

\noindent The system $(u_1,\dots, u_N)$ of unit vectors in $\mathbb{K}^m$ is called an {\it equiangular tight frame} ETF$(m,N)$ if it is tight and
the value of $|\langle u_i, u_j\rangle|$ is constant over all $i\neq j.$
It is well known (see e.g. \cite[Theorem 5.7]{FR} ) that if $(u_1,\ldots u_N )$ is an ETF$(m,N)$ then
\begin{equation}\label{Welch Bound}
|\langle u_i,u_j \rangle|=\sqrt{\frac{N-m}{m(N-1)}}
\qquad \textrm{ for all } i,j\in \{1,\dots, N \},\;i\neq j,
\end{equation}
and the constant $\alpha$ is also determined and is equal to $\frac{m}{N}.$
Notice that the existence of an equiangular tight frame consisting of $N$ unit vectors in $\mathbb{K}^m$ is equivalent to the existence of a Parseval frame such that
\begin{equation*}
    (U^*U)_{ii}=\frac{m}{N} \; \textrm{ and }\; |U^*U|_{ij}=\frac{m}{N}\sqrt{\frac{N-m}{m(N-1)}} \;\textrm{ for }\; i\neq j,
\end{equation*}
where $i,j\in \{1,\dots, N\}.$ Moreover, the quantity $\mu_{\mathbb{K}}(U)=\lambda_{\mathbb{K}}(U)$ meets the upper bound for the maximal relative projection constant, given in \cite{KLL}. We present it in the form stated in \cite [ Theorem 5 ]{ FS}, where also an easier proof of this result was provided.

\begin{thm}\label{SF}

For integers $N\geq m$, the maximal relative projection constant $\lambda_{\mathbb{K}}(m,N)$ is upper bounded by
$$
\delta_{m,N} := \frac{m}{N} \left( 1 + \sqrt{\frac{(N-1)(N-m)}{m}} \right).
$$
Moreover, the following properties are equivalent:
\begin{enumerate}[i)]
\item There is an equiangular tight frame consisting of $N$ vectors in $\mathbb{K}^m,$
\item $\mu_\mathbb{K}(m,N)=\tfrac{m}{N}\left(1 +\sqrt{\tfrac{(N-1)(N-m)}{m}} \right),$
\item $\lambda_\mathbb{K}(m,N)=\tfrac{m}{N}\left(1 +\sqrt{\tfrac{(N-1)(N-m)}{m}} \right).$
\end{enumerate}
\end{thm}

In 1994, H. K\"onig and N. Tomczak-Jaegermann stated the following estimation.
\begin{thm}[stated in \cite{KT}; proved in \cite{BB2}]\label{nMPC}
Let $ m>1$ then
\begin{enumerate}[i)]
    \item $\lambda_\mathbb{R}(m) \leq \frac{2}{m+1}\left(1+\frac{m-1}{2}\sqrt{m+2}\right)\,$ 
    \item $ \lambda_\mathbb{C}(m)\leq \frac{1}{m}\left(1+(m-1)\sqrt{m+1}\right).$ 
\end{enumerate}
\end{thm}
\noindent Unfortunately, their proof is based on an erroneous lemma, as was pointed out in \cite{CLe}.  Recently, using methods similar to those in \cite{BC}, we have proved the latter.
Observe that the upper bound in Theorem \ref{nMPC} is equal to $\delta_{m, \frac{m(m+1)}{2}} $ in the real case and $\delta_{m, m^2} $ in the complex case. The number of vectors in an ETF cannot exceed $\frac{m(m+1)}{2}$ in the real case and $m^2$ in the complex case (see, e.g., \cite[Theorem 5.10]{FR}). So if there exists an ETF with the maximum possible number of vectors ({\it the maximal ETF}), the upper bound is realized. 

\begin{thm}\label{jakis}
Let $ m>1.$
\begin{enumerate}[i)]
    \item  If there exists a maximal ETF in $\mathbb{R}^m$ then $\lambda_\mathbb{R}(m) = \frac{2}{m+1}\left(1+\frac{m-1}{2}\sqrt{m+2}\right)\,$ 
    \item  If there exists a maximal ETF in $\mathbb{C}^m$ then $ \lambda_\mathbb{C}(m) = \frac{1}{m}\left(1+(m-1)\sqrt{m+1}\right).$ 
\end{enumerate}
\end{thm}
There are numerous examples of complex maximal ETFs, for example, for $m \in\{1,\ldots, 53\}$ (see, e.g., \cite{FM1}). 
In fact, it is conjectured that there is a complex maximal ETF in every dimension (Zauner's conjecture \cite{Z}).  In view of Theorem \ref{jakis}, Zauner's Conjecture implies the following hypothesis.  
\begin{conjecture}
For every $m\geq1$
$$
\lambda_\mathbb{C}(m)=\frac{1}{m}\left(1+(m-1)\sqrt{m+1}\right).
$$
\end{conjecture}
Unlike in the complex case, real maximal ETFs are rare objects. The only known cases are for $m$ equal to $2,$ $3,$ $7$ and $23.$ Therefore, we have the following.
\begin{thm}\label{MPC}
~
\begin{itemize}
    \item $\lambda_\mathbb{R}(2)=\frac{4}{3}.$ 
    \item $\lambda_\mathbb{R}(3)=\frac{1+\sqrt{5}}{2};$ 
    \item $\lambda_\mathbb{R}(7)=\frac{5}{2};$
    \item $\lambda_\mathbb{R}(23)=\frac{14}{3}.$
    \end{itemize}
    \end{thm}
Many of the community believe these are all real cases where maximal ETFs exist.
Since real maximal ETFs exist only in a handful of dimensions, it is natural to seek more general structures. In the remainder of this paper, we investigate biangular tight frames, and spherical designs as candidates for extremal configurations, which naturally arise in this context.

\section{ Biangular tight frames.}
A system of  vectors $U=(u_j)_{j=1}^N$ with equal norms in $\mathbb{K}^m$ is called {\it biangular} if there exist $b>c\geq 0$ such that 
   $$
   |\langle u_i, u_j\rangle|\in\{b,c\}\;\;\; \textrm{ for all }\; i\neq j\in[1:N]. 
   $$
As it was shown in \cite[Proposition 5.1]{Cas} every biangular, tight frame for $\mathbb{K}^m$ is equidistributed i.e., there exist $k\in[1:N]$ such that
$$
\#\{j\in[1:N]:\; |\langle u_i,\; u_j\rangle|=b\}=k \;\;\; \textrm{ for all }\; i\in[1:N].
$$
The above equidistribution property implies that biangular tight frames share an important property with ETFs: namely, the quantities $\mu_{\mathbb K}(U)$ and $\lambda_{\mathbb K}(U)$ coincide.
\begin{lemma}\label{biangularmu}
Let $U\in \mathbb{K}^{m\times N}$ be a biangular Parseval tight frame. Then $$\lambda_{\mathbb{K}}(U)=\mu_{\mathbb{K}}(U).$$
\end{lemma}
{\sc Proof.} Since $U$ is a biangular tight frame, there exist $a,b,c\in \mathbb{R}$, 
such that $(U^*U)_{ii}=a$ for $i \in [1:N] $ and in every row of $|U^*U|$ exactly $k$ entries is equal to $b$ and exactly $N-k-1$ entries is equal to $c$ for some $k\in[1:N-2].$ Let us define
$$
J^{b}:=\{(i,j)\in[1:N]^2: |U^*U|_{ij}=b \textrm{ and } i\neq j\}
$$
and
$$
J_l^b:=\{(i,j)\in J^b:  i= l \;\;{\rm or }\;\; j=l\},
$$
where $l\in [1:N].$ 
Since $|U^*U|$ is symmetric, we have 
$$
\#J_l^b=\#\{(l,j)\in J^b\}+\#\{(i,l)\in J^b\}=2\#\{(l,j)\in J^b\}=2k.
$$
Analogously, define the sets $J^c,J^c_l.$  Then
$$\#J_l^c=2(N-k-1).$$
Consequently, for every unit vector $t\in \mathbb{R}^N_+$
\begin{align*}
    \sum_{i,j=1}^Nt_it_j|U^*U|_{ij}&= a\sum_{i=1}^N t_i^2+
    b\sum_{\substack{ \\ (i,j)\in J^b}}t_it_j + c\sum_{\substack{ \\ (i,j)\in J^c}}t_it_j\\
    &\leq
    a+
    \tfrac{1}{2}b\sum_{\substack{ (i,j)\in J^b}}(t_i^2+t_j^2) + \tfrac{1}{2}c\sum_{\substack{ \\ (i,j)\in J^c}}(t_i^2+t_j^2) \\
    &= a+ \tfrac{1}{2}b\sum_{l=1}^N\sum_{(i,j)\in J_l^b} t_l^2 +\tfrac{1}{2}c\sum_{l=1}^N\sum_{(i,j)\in J_l^c} t_l^2\\
    &= a+ kb\sum_{l=1}^N t_l^2 +(N-k-1)c\sum_{l=1}^N t_l^2\\
    &= a+ kb +(N-k-1)c=\mu_\mathbb{K}(U),
\end{align*}
Taking the supremum over all $t\in\mathbb R_+^N$ with $\|t\|_2=1$ yields
\[
\lambda_{\mathbb K}(U)\le \mu_{\mathbb K}(U).
\]
The reverse inequality is immediate from the definitions by taking
$t=\frac1{\sqrt N}(1,\ldots,1)$.
Hence,
\[
\lambda_{\mathbb K}(U)=\mu_{\mathbb K}(U),
\]
as required.\\
~\\
Observe that the sequence $(\lambda_{\mathbb K}(m,N))_{N\ge m}$ is nondecreasing with respect to $N$, and
\[
\lambda_{\mathbb K}(m)=\sup_{N\ge m}\lambda_{\mathbb K}(m,N),
\]
Moreover, in all currently known real cases where the maximal absolute projection constant is attained, the corresponding maximal relative projection constant is realized by a maximal equiangular tight frame with $N=\frac{m(m+1)}{2}$ vectors. This provides a strong motivation for studying highly structured systems containing at least this many vectors. In particular, the midpoints of a regular simplex form a biangular tight frame with exactly $\frac{m(m+1)}{2}$ vectors, making them natural candidates for extremal configurations.

\begin{example}
Let $m\geq 3.$ The midpoints of the edges of a regular simplex in $\mathbb{R}^m$ can be given by
\begin{align*}
    v_{i,k}&:=\frac{1}{\sqrt{m-1}}\left(e_i+e_k-\frac{2}{m}\left(1+\frac{1}{\sqrt{m+1}}\right)\mathbb{1}_m\right) \;\; \textrm{ for }\;\; i<k\in[1:m]\\
    v_{i,i}&:=\frac{1}{\sqrt{m-1}}\left(e_i + \frac{1}{m}\left(\frac{m-1}{\sqrt{m+1}}-1\right)\mathbb{1}_m\right) \;\; \textrm{ for }\;\; i\in[1:m]
\end{align*}
\end{example}
where $\{e_i\}_{i=1}^m$ is the canonical basis of $\mathbb{R}^m,$ and $\mathbb{1}_m\in \mathbb{R}^m$ is the vector whose entries are all equal to $1.$
The following properties can be verified by direct computation. In particular, all vectors have the same norm, and the absolute values of their pairwise inner products take only two distinct values. More precisely,
\begin{align*}
    \|v_{i,k}\|^2&=\frac{2}{m+1} \;\;\textrm{ for }\;\; i\leq k\in[1:m]\\
    |\langle v_{i,i},v_{j,j}\rangle| &= \frac{m-3}{m^2-1} \;\;\textrm{ for }\;\; i\neq j\in[1:m]\\
     |\langle v_{i,k},v_{j,l}\rangle| &= \frac{m-3}{m^2-1} \;\;\textrm{ for }\;\; \# \{i,k\}\cap\{j,l\}=1 \;\textrm{ and }\; (i,k)\neq(j,l)\\
     |\langle v_{i,k},v_{j,l}\rangle| &= \frac{4}{m^2-1} \;\;\textrm{ for other cases. }
\end{align*}
Furthermore, $\sum_{i\leq k}\|v_{i,k}\|^2 =m$ and
\begin{align*}
\sum_{i\leq k}\sum_{j\leq l}|\langle v_{i,k},v_{j,l}\rangle|^2&=\frac{m(m+1)}{2}\cdot\frac{4}{(m+1)^2}+m(m^2-1)\cdot\frac{(m-3)^2}{(m^2-1)^2}\\
&+\frac{m(m^2-1)(m-2)}{4}\cdot\frac{16}{(m^2-1)^2}=m=\frac{1}{m}\left(\sum_{i\leq j}\|v_{i,k}\|^2\right)^2.   
\end{align*}
   Consequently, the family $V_m=(v_{i,k})\in\mathbb{R}^{m\times\frac{m(m+1)}{2}}$ is a biangular Parseval frame, and by Lemma \ref{biangularmu}
$$
\lambda_{\mathbb{R}}(V_m)=\mu_{\mathbb{R}}(V_m)=\frac{4(m-2)}{m+1}.
$$
As an immediate consequence of the preceding example, we obtain the following lower bound for the maximal absolute projection constant.
\begin{cor}
For every $m\geq 3$ we have
$$\lambda_{\mathbb{R}}(m)\geq \frac{4(m-2)}{m+1}.$$
\end{cor}

\begin{figure}[H]
\begin{center}
\includegraphics[width=8cm]{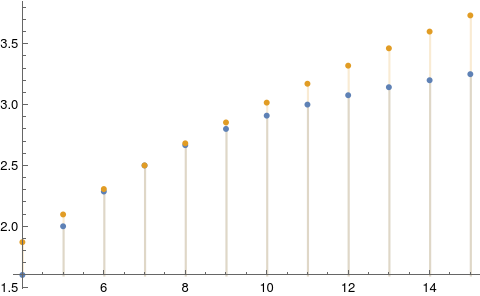}   
\caption{ Yellow dots show upper bound for $\lambda_{\mathbb{R}}(m),$ blue ones show value of $\lambda_{\mathbb{R}}(V_m)$}\label{wykr} 
\end{center}
\end{figure}
For $m=7,$ these vectors form the maximal equiangular tight frame and $\lambda_{\mathbb{R}}(V_7)=\lambda_{\mathbb{R}}(7).$ Moreover, as illustrated in Figure~\ref{wykr}, the gap between $\lambda_{\mathbb{R}}(V_m)$ and the upper bound from Theorem~\ref{nMPC} is very small for $m=6$ and $m=8$. These values also agree with the numerical estimates obtained by B.~L.~Chalmers and, more recently, by V.~Sivashankar, Q.~Tang, and T.~Wakhare (see \cite{STW}), providing further evidence for the following conjecture.
\begin{conjecture}
~
 \begin{itemize}
    \item $\lambda_\mathbb{R}(6)=\frac{16}{7}.$ 
    \item $\lambda_\mathbb{R}(8)=\frac{8}{3};$ 
    \end{itemize}
   
\end{conjecture}
\section{Spherical $(2,2)$-designs.}
There are several equivalent characterizations of spherical designs (see, e.g., \cite{Ba,DGS}). In what follows, we consider weighted spherical $(t,t)$-designs. For our purposes, it is convenient to define them as configurations of unit vectors ($U=(u_j)_{j=1}^N$), together with nonnegative weights ($w=(w_j)_{j=1}^N$), for which equality holds in the following theorem.
   
\begin{thm}[\cite{HW}]\label{design}
 Let    $U=(u_j)_{j=1}^N$ be a sequence of unit vectors in  $\mathbb{K}^m,$ and $(w_j)_{j=1}^N $ be nonnegative weights, i.e.,  $w_j\geq 0,$ $\sum_{j=1}^N w_j=1.$ Then 
\begin{equation}
    \sum_{i,j=1}^N w_iw_j|\langle u_i, u_j\rangle |^{2t}\geq c_t(m,\mathbb{K}),
\end{equation}
where
\begin{equation*}
    c_t(m,\mathbb{R})=\frac{1\cdot 3 \cdot 5 \cdots (2t-1)}{m(m+2)\cdots(m+2(t-1))},\;\;
     c_t(m,\mathbb{C})=\frac{1}{\binom{m+t-1}{t} }.
\end{equation*}
\end{thm}
Using the above theorem for $t=2$, we can shorten the proof of Theorem \ref{nMPC} given in the paper \cite{BB2}.
\begin{thm}
Let $1<m\leq N.$ Then the following inequalities hold
$$
\lambda_\mathbb{R}(m,N)\leq \frac{2}{m+1}\left(1+\frac{m-1}{2}\sqrt{m+2}\right)
$$
$$
\lambda_\mathbb{C}(m,N)\leq \frac{1}{m}\left(1+(m-1)\sqrt{m+1}\right)
$$
\end{thm}
{\sc Proof.} Let $U\in\mathbb{K}^{m\times N}$ be such that $UU^*={\rm I}_m$. Denote by $u_i$  the i-th column of the matrix $U.$ Observe that for matrix $\widetilde{U}\in\mathbb{K}^{m\times \widetilde{N}}$ created from only nonzero columns (since $u_1,\dots, u_N$ form a tight frame, they span $\mathbb{K}^m$ and $\widetilde{N}\geq m$) we have
$$
\sum_{i,j=1}^{N}t_it_j|U^* U|_{ij}\leq \sum_{i,j=1}^{\widetilde{N}}t_it_j|\widetilde{U}^* \widetilde{U}|_{ij}  \;\textrm{ and }\; \widetilde{U}\widetilde{U}^*={\rm I}_m\; .
$$
So, without loss of generality, we can assume that all $u_i$ are nonzero vectors.
Combining the Cauchy-Schwarz inequality and   Theorem \ref{design}  we get
\begin{align*}
\sum_{i,j=1}^N &t_it_j\frac{(|\langle u_i,u_j\rangle|-\varphi
\|u_i\|\|u_j\|)^2}{\|u_i\|\|u_j\|}
=\sum_{i,j=1}^N t_it_j\frac{(|\langle u_i,u_j\rangle|^2-\varphi^2\|u_i\|^2\|u_j\|^2)^2}{\|u_i\|\|u_j\|(|\langle u_i,u_j\rangle|+\varphi\|u_i\|\|u_j\|)^2}\\
&\geq \sum_{i,j=1}^Nt_it_j \frac{(|\langle u_i,u_j\rangle|^2-\varphi^2\|u_i\|^2\|u_j\|^2)^2}{(1+\varphi)^2\|u_i\|^3\|u_j\|^3} = \left(\sum_{i,j=1}^N \left|\left\langle \frac{u_i}{\|u_i\|}, \frac{u_j}{\|u_j\|} \right\rangle\right|^4\|t_iu_i\|\|t_ju_j\|\right.\\
&\left.-2\varphi^2 \sum_{i,j=1}^Nt_it_j\frac{|\langle u_i, u_j \rangle|^2}{\|u_i\|\|u_j\|}+\varphi^4\sum_{i,j=1}^Nt_it_j\|u_i\|\|u_j\|\right)(1+\varphi)^{-2}  \\
&\geq
\left(c_2(m,\mathbb{K})\big(\sum_{i=1}^Nt_i\|u_i\|\big)^2 
-2\varphi^2 \sum_{i,j=1}^Nt_it_j\frac{|\langle u_i, u_j \rangle|^2}{\|u_i\|\|u_j\|}+\varphi^4 \big(\sum_{i=1}^Nt_i\|u_i\|\big)^2\right)(1+\varphi)^{-2}
\end{align*}
Squaring the addends of the left-sided sum and rearranging the latter inequality gives 
\begin{align}\label{inq2}\notag
    2\varphi\sum_{i,j=1}^{N}t_it_j|U^*U|_{ij} &\leq \Big(1+\frac{2\varphi^2}{(1+\varphi)^2}\Big)\sum_{i,j=1}^Nt_it_j\frac{|\langle u_i, u_j \rangle|^2}{\|u_i\|\|u_j\|} \\
    &+
    \big(\sum_{i=1}^Nt_i\|u_i\|\big)^2\left(\varphi^2-\frac{c_2(m,\mathbb{K})}{(1+\varphi)^2}-\frac{\varphi^4}{(1+\varphi)^2}\right)
\end{align}
Now observe that using the Cauchy-Schwarz inequality and the tightness of the vectors $u_1,\dots,u_N$, we have
\begin{align*}
    \sum_{i,j=1}^{N}t_it_j\frac{|\langle u_i, u_j\rangle|^2}{\|u_i\|\|u_j\|}&\leq \sqrt{\sum_{i,j=1}^{N}\frac{t_i^2}{\|u_i\|^2}|\langle u_i, u_j\rangle|^2}\sqrt{\sum_{i,j=1}^{N}\frac{t_j^2}{\|u_j\|^2}|\langle u_i, u_j\rangle|^2}\\
    &=\sum_{i=1}^N\frac{t_i^2}{\|u_i\|^2}\sum_{j=1}^{N}|\langle u_i,u_j\rangle|^2
    =\sum_{i=1}^Nt_i^2=1
\end{align*}
and
\begin{align*}
\left(\sum_{i=1}^{N}t_i\|u_i\|\right)^2=\langle[t_1,\dots,t_N], [\|u_1\|,\dots, \|u_N\|]\rangle^2 &\leq \sum_{i=1}^Nt_i^2\sum_{i=1}^{N}\|u_i\|^2 = {\rm tr}(U^*U)\\ 
&=\;{\rm tr}(UU^*)=m.
\end{align*}
In view of the above, if we take $\varphi>0,$ such that  $\varphi^2-\frac{c_2(m,\mathbb{K})}{(1+\varphi)^2}-\frac{\varphi^4}{(1+\varphi)^2}\geq 0,$ \eqref{inq2} reads
\begin{equation*}\label{inq2u}
     2\varphi\sum_{i,j=1}^{N}t_it_j|U^*U|_{ij} \leq \Big(1+\frac{2\varphi^2}{(1+\varphi)^2}\Big)+
    m\left(\varphi^2-\frac{c_2(m,\mathbb{K})}{(1+\varphi)^2}-\frac{\varphi^4}{(1+\varphi)^2}\right)
\end{equation*}
Setting $\varphi=\frac{1}{\sqrt{m+2}}$ in the real case and  $\varphi=\frac{1}{\sqrt{m+1}}$ in the complex case, we obtain the desired upper bound of $\lambda_{\mathbb{K}}(m,N)$\\

A straightforward calculation shows that maximal ETFs are also spherical $(2,2)$-designs (i.e., a weighted spherical $(2,2)$-design with equal weights). Moreover, the following theorem implies that they have the smallest possible number of vectors among all spherical $(2,2)$-designs.

\begin{thm}[\cite{Wal}]
Let $n$  be the number of vectors in a weighted spherical $(t,t)$-design then
$$
 { m-1+t \choose t}\leq n \leq{ m-1+2t \choose 2t}.
$$
 \end{thm}
 Contrary to the real maximal ETFs, spherical $(t,t)$-designs exist in every dimension, shown in \cite{SZ}. 
~\\ 
 Another relevant example of a weighted spherical $(2,2)$-design with the smallest possible number of vectors is a system of $16$ vectors in $\mathbb{R}^5$  (see \cite{HW}, {\it Example 3.5}). By applying {\it Example 5.3 } from \cite{Wal1}, this configuration gives rise to a Parseval tight frame $U.$ Interestingly, the corresponding value of $\mu_{\mathbb{R}}(U)$ is smaller than the conjectured value $\mu_{\mathbb{R}}(5,16)$. Moreover, numerical optimization suggests that
$\lambda_{\mathbb{R}}(U) \approx 2.06615,$ which is also below the conjectured value of $\lambda_{\mathbb{R}}(5,16)\approx 2.06919.$
In addition, this frame has the same sign matrix of pairwise inner products as the frame conjectured to maximize $\lambda_{\mathbb{R}}(5,16)$ (see, \cite{DFFL}  and \cite{DLL}, {\it Example 4.1}).

This striking coincidence naturally raises the question of whether the sign matrix, rather than the precise values of the inner products, is the key object governing extremal projection constants, which is particularly important in view of an alternative expression for the maximal absolute projection constant. Before giving this expression, we need to elucidate some notation.  

The eigenvalues of a self-adjoint matrix $M$  arranged in nonincreasing order  will be denoted by $\lambda_1^{\downarrow}(M),$  $\lambda^{\downarrow}_2(M),$ 
 $\dots,$ $\lambda^{\downarrow}_N(M).$  The set of $N\times N$ Seidel matrices will be denoted by $\mathcal{S}^{N\times N}.$ Let us recall that a self-adjoint matrix $B\in \mathbb{K}^{N\times N}$ with $B_{i,i}=0$ for all $i$ and $|B_{i,j}|=1$ for all $i\neq j$ is called the Seidel matrix. It is worth mentioning that in the real case these matrices are called Seidel adjacency matrices and have their origin in graph theory. Now, we are ready to state the following.

\begin{thm}[\cite{CLe}, Theorem 2.1]\label{lammbda2}
For integers $N \ge m$, we have
\begin{align}\label{lambda2}
    \lambda_{\mathbb{K}}(m,N)=\max\bigg\lbrace \sum_{k=1}^{m}\lambda_k^{\downarrow}(TAT):\; T=&{\rm diag}(t),\;t\in\mathbb{R}_+^N,\;\|t\|=1,\notag\\ 
    & A=I_N+B, \; B\in \mathcal{S}_\mathbb{K}^{N\times N}\bigg\rbrace.
\end{align}
\end{thm}

\noindent 
Using the same notation we can also rephrase the definition of   quasimaximal relative projection constant as follows 
\begin{equation}\label{mi2}
    \mu_\mathbb{K}(m,N)=\max\bigg\lbrace \frac{1}{N} \sum_{k=1}^{m}\lambda_k^{\downarrow}(A):\;  A=I_N+B, \; B\in \mathcal{S}_\mathbb{K}^{N\times N}\bigg\rbrace.
\end{equation}
A nice connection exists between maximizers of both expressions for maximal relative projection constants. If matrix $U$ achieves the maximum in \eqref{lambda1}, then matrix $A={\rm sgn}(U^\top U)$ realizes the maximum in \eqref{lambda2} (the same holds for quasimaximal relative projection constants).\\
~\\
Another example supporting the connection between minimal weighted spherical $(2,2)$-designs and projection constants comes from dimension $4$. Numerical computations of B.~L.~Chalmers suggest that
\[
\lambda_{\mathbb{R}}(4)=\lambda_{\mathbb{R}}(4,11)\approx 1.85008.
\]
On the other hand, Park \cite{Park} constructed the Gram matrix of a weighted spherical $(2,2)$-design consisting of $11$ vectors in $\mathbb{R}^4$. It is conjectured that no weighted spherical $(2,2)$-design in $\mathbb{R}^4$ exists with fewer vectors (see, e.g., \cite{Wal}). Fixing the sign matrix in the optimization problem of Theorem~\ref{lammbda2} to be the sign matrix of Park's Gram matrix, we numerically obtain the value
$1.85008,$ which coincides with Chalmers' conjectured value of $\lambda_{\mathbb{R}}(4)$ and with the numerical estimates obtained in \cite{STW}.

This example provides further evidence that weighted spherical $(2,2)$-designs with the smallest possible number of vectors may play a fundamental role in determining maximal projection constants.

In the first version of this preprint, the examples above led us to formulate the following conjecture.

\begin{conjecture}\label{conj:designsign}
Let $U=(u_j)_{j=1}^N$ be a weighted spherical $(2,2)$-design in $\mathbb{R}^m$. If no two distinct vectors are orthogonal, equivalently,
$$
{\rm sgn}(U^\top U)-I_N\in\mathcal{S}_\mathbb{R}^{N\times N},
$$
then
$$
\lambda_m({\rm sgn}(U^\top U))
=\lambda_\mathbb{R}(m,N)
=\lambda_\mathbb{R}(m).
$$
\end{conjecture}

Conjecture~\ref{conj:designsign} is false. Hu and Tang \cite{HT} constructed an explicit counterexample consisting of seven unit vectors in $\mathbb{R}^3$ with positive weights. Their configuration is a weighted spherical $(2,2)$-design and contains no orthogonal pair, so it satisfies all the assumptions of Conjecture~\ref{conj:designsign}. Nevertheless, for
$$
A={\rm sgn}(U^\top U)
$$
they proved that
\begin{equation}
\lambda_3(A)=\frac{7+2\sqrt{13}}{9}
<\lambda_\mathbb{R}(3,7)
=\lambda_\mathbb{R}(3)
=\frac{1+\sqrt5}{2}.
\end{equation}
The example has a particularly simple geometric form: one vector is a pole and the remaining six form a regular hexagon on a parallel circle.

Thus the sign pattern of an arbitrary weighted spherical $(2,2)$-design is not sufficient, by itself, to force extremality. It is worth noting, however, that the counterexample has seven vectors, whereas in dimension $3$ a maximal ETF already gives a spherical $(2,2)$-design with six vectors. The examples discussed above therefore still suggest that designs with the smallest possible, or nearly smallest possible, number of vectors may be more rigid. This leads to the following weaker question.

\begin{question}
Let $U=(u_j)_{j=1}^N$ be a weighted spherical $(2,2)$-design in $\mathbb{R}^m$ with no two distinct vectors orthogonal, and suppose that $N$ is minimal among weighted spherical $(2,2)$-designs in $\mathbb{R}^m$. Under what additional assumptions, if any, does
$$
\lambda_m({\rm sgn}(U^\top U))=\lambda_\mathbb{R}(m,N)=\lambda_\mathbb{R}(m)
$$
hold?
\end{question}

\section*{Acknowledgments}
We thank Tao Hu and Quanyu Tang for informing us of their counterexample and for helpful correspondence concerning the conjecture.


\begin{thebibliography}{00}
\bibitem{Ba} E. Bannai, E. Bannai, {\it A survey on spherical designs and algebraic combinatorics on spheres,} Eur. J. Comb. 30(6) (2009) 1392–1425.
\bibitem{B} G. Basso, {\it Computation of maximal projection constants,} J. Funct. Anal. 277/10 (2019), 3560--3585.
\bibitem{BC} B. Bukh, C. Cox, {\it Nearly orthogonal vectors and small antipodal spherical codes,} Isr. J. Math. 238, 359–388 (2020). doi: 10.1007/s11856-020-2027-7
\bibitem{CHM}  B. Chalmers and F. Metcalf, {\it Determination of a minimal projection from $\mathcal{C}[-1,1]$ onto the quadratics,} Numerical Functional Analysis and Optimization, 11(1-2), 1-10, 1990.
\bibitem{CLe} B. L. Chalmers, G. Lewicki, \textit{Three-dimensional subspace of $\ell_{\infty}^{(5)}$ with maximal projection constant,} J. Funct. Anal. 257/2 (2009), 553--592.
\bibitem{CL} B. L. Chalmers, G. Lewicki, {\it A proof of the Gr\"unbaum conjecture,} Studia Math. 200 (2010), 103--129.
\bibitem{Cas} P.G. Casazza, A. Farzannia, J.I. Haas, T.T. Tran, {\it 
Toward the classification of biangular harmonic frames,}
Applied and Computational Harmonic Analysis, 46(3), 2019,  544-568.

\bibitem{Del1} P. Delsarte, {\it Bounds for unrestricted codes, by linear programming,} Philips Res. Rep., 27:272–289, 1972.

\bibitem{DGS} P. Delsarte, J.M. Goethals, J.J. Seidel, {\it Spherical codes and designs,} Geom. Dedic. 6 (3) (1977)

\bibitem {DFFL} B. Deregowska, M. Fickus, S. Foucart, B. Lewandowska,
{\it On the value of the fifth maximal projection constant,}
J. Funct. Anal.283/10, (2022)
\bibitem{BB2} B. Deregowska,  B. Lewandowska, {\it  A simple proof of the Grünbaum conjecture,}  J. Funct. Anal. 285 (2023), Art no.: 109950

\bibitem {DLL}B. Deregowska, B. Lewandowska, G. Lewicki, {\it Minimal projections onto subspaces generated by sign-matrices,} J. Approx. Theory 304 (2024)
\bibitem{FM1} M. Fickus, D. Mixon,  {\it Tables of the existence of equiangular tight frames,} arXiv:1504.00253v2 (2016). doi: 10.48550/arXiv.1504.00253
\bibitem{FR} S. Foucart, H. Rauhut, {\it A Mathematical Introduction to Compressive Sensing,} Birkh\"auser, 2013.
\bibitem{FS}S. Foucart, L. Skrzypek, {\it On maximal relative projection constants,} J. Math. Anal. Appl. 447/1 (2017), 309--328.
Journal of Approximation Theory, 235, 74-91, 2018.
\bibitem{HT} T~Hu and Q.~Tang, {\it A seven-vector counterexample to a conjecture on maximal projection constants}, preprint (2026), https://doi.org/10.13140/RG.2.2.33161.20325

\bibitem{HW} D. Hughes, S. Waldron,
{\it Spherical $(t,t)-$designs with a small number of vectors,},
Linear Algebra and its Applications, 608, (2021), 84-106.


\bibitem{KS} I. M. Kadec, M. G. Snobar, {\it Certain functionals on the Minkowski compactum,} Math. Notes 10 (1971), 694--696 (English transl.).
\bibitem{K} H. K\"{o}nig, {\it Spaces with large projection constants,} Isr. J. Math. 50/3 (1985), 181--188.
\bibitem{KLL} H. K\"{o}nig, D. Lewis, P.-K. Lin,
{\it  Finite dimensional projection constants,}
Studia Mathematica 75/3 (1983), 341--358.
\bibitem{KT} H. K\"{o}nig,  N. Tomczak-Jaegermann,
{\it Norms of minimal projections,} 
J. Funct. Anal. 119/2 (1994),  253--280.
\bibitem{G} B. Gr\"unbaum, {\it Projection constants,} Trans. Amer. Math. Soc. 95 (1960), 451--465. 

\bibitem{Park} J. Park {\it Interacion energies, lattices, and designs} Ph.D. Thesis. Georgia Institute of Technology, 2020.
\bibitem{SG} A. J. Scott, M. Grassl, {\it Symmetric informationally complete positive-operator-valued measures: a new computer study,} Journal of Mathematical Physics 51/4 (2010): 042203.
\bibitem{SZ} P.D. Seymour, T. Zaslavsky, {\it Averaging sets: a generalization of mean values and
spherical designs,} Adv. Math. 52(3), 213–240 (1984)
\bibitem{STW} V. Sivashankar, Q. Tang, T. Wakhare, {\it Graph Eigenvalues and Projection Constants,} arXiv:2608.02429 (2026), doi: 10.48550/arXiv.2608.02429.
\bibitem{Wal1} S. Waldron, {\it A Sharpening of the Welch Bounds and the Existence of Real and Complex Spherical $t$–Designs,} in IEEE Transactions on Information Theory, vol. 63, no. 11, pp. 6849-6857, Nov. 2017, doi: 10.1109/TIT.2017.2696020.
\bibitem{Wal} S. Waldron, {\it An introduction to finite tight frames}, Applied and Computational Harmonic Analysis, Birhauser, 2018.
\bibitem{W} P. Wojtaszczyk, {\it Banach Spaces for Analysts,} Cambridge University Press, Cambridge, 1991.
\bibitem{Z} G. Zauner, {\it Quantenredgns: Grundzüge einer nichtkommutativen Designtheorie}, Ph.D. Thesis, University of Vienna, 1999. 

\end{thebibliography}
\end{document}